\documentclass{article}

\usepackage{amsmath,amssymb,amsthm}
\usepackage{tikz}
\usepackage{color}
\usepackage[toc]{appendix}
\usepackage{graphicx}
\usepackage{fancyhdr}
\usepackage{enumitem}
\usepackage{bbm}
\usepackage{parskip}
\usepackage{float}
\usepackage{chngpage}
\usepackage{calc}
\usepackage{bigints}
\usepackage{array}
\usepackage{booktabs}
\usepackage{rotating}
\usepackage{multirow}
\usepackage{adjustbox}
\usepackage{tabularx}
\usepackage{verbatim}
\usepackage{mathtools}
\usepackage{ragged2e}
\usepackage[makeroom]{cancel}
\usepackage{caption}
\usepackage{hyperref}
\usepackage{caption}
\usepackage{subcaption}
\usepackage{appendix}
\usepackage{pgfplots}
\pgfplotsset{compat=1.16}
\usepackage[english]{babel}
\usepackage{hyphenat}
\usepackage[makeindex]{imakeidx}
\usepackage{xcolor}

\usepackage{amsmath}

\usetikzlibrary{datavisualization}
\usetikzlibrary{matrix}
\usetikzlibrary{datavisualization.formats.functions}

\newtheorem{theorem}{Theorem}[section]

\newtheorem{lemma}[theorem]{Lemma}

\newtheorem{remark}[theorem]{Remark}

\makeatletter
\def\section{\@startsection {section}{1}{\z@}{3.25ex plus 1ex minus
		.2ex}{1.5ex plus .2ex}{\large\bf}}
\def\subsection{\@startsection{subsection}{2}{\z@}{3.25ex plus 1ex minus
		.2ex}{1.5ex plus .2ex}{\normalsize\bf}}
\@addtoreset{equation}{section} 
\makeatother

\title{A new look to branching Brownian motion from a particle-based
	reaction–diffusion dynamics point of view}
\author{Alberto Lanconelli\thanks{Dipartimento di Scienze Statistiche Paolo Fortunati, Università di Bologna, Bologna, Italy. \textbf{e-mail}: alberto.lanconelli2@unibo.it} \and Berk Tan Perçin\thanks{Dipartimento di Scienze Statistiche Paolo Fortunati, Università di Bologna, Bologna, Italy. \textbf{e-mail}: berktan.percin2@unibo.it}}
\date{\today}

\begin{document}
	
	\maketitle
	
	\bigskip
	
	\begin{abstract}
		Aim of this note is to analyse branching Brownian motion within the class of models introduced in the recent paper \cite{delRazo} and called chemical diffusion master equations. These models provide a description for the probabilistic evolution of chemical reaction kinetics associated with spatial diffusion of individual particles. We derive an infinite system of Fokker-Planck equations that rules the probabilistic evolution of the single particles generated by the branching mechanism and analyse its properties using Malliavin Calculus techniques, following the ideas proposed in \cite{Lanconelli_CDME}. Another key ingredient of our approach is the McKean representation for the solution of the Fisher-Kolmogorov-Petrovskii-Piskunov equation and a stochastic counterpart of that equation. We also derive the reaction-diffusion partial differential equation solved by the average concentration field of the branching Brownian system of particles. 
	\end{abstract}
	
	Key words and phrases: Branching Brownian Motion, particle-based reaction-diffusion dynamics, Malliavin calculus. 
	
	AMS 2020 classification: 60H07; 60H30; 92E20.
	
	\allowdisplaybreaks
	
	\bigskip
	
	\section{Introduction and statement of the main results}
	
	Let $\{X_k(t),k\in\{1,...,\mathtt{n}(t)\}\}_{t\geq 0}$ be a (binary) branching Brownian motion. That means, at time zero a single particle $\{X(t)\}_{t\geq 0}$ starting at the origin begins to perform Brownian motion in $\mathbb{R}$; after an exponential time $\tau$ of parameter one, the particle splits into two identical, independent copies of itself that start Brownian motion at $X(\tau)$. This process is repeated ad infinitum, producing a collection of $\mathtt{n}(t)$ particles $\{X_1(t),X_2(t),...,X_{\mathtt{n}(t)}(t)\}_{t\geq 0}$ (we refer the reader to \cite{BBM} and the references quoted there for more details on the subject). For $t\geq 0$, $n\geq 1$ and $A\in\mathcal{B}(\mathbb{R}^n)$ we set
	\begin{align}\label{def rho}
	\int_{A}\rho_n(t,y_1,...,y_n)dy_1\cdot\cdot\cdot dy_n&:=\mathbb{P}\left(\{\mathtt{n}(t)=n\}\cap\{(X_1(t),...,X_n(t))\in A\}\right);
	\end{align}
	notice that by construction the sequence $\{\rho_n\}_{n\geq 1}$ fulfils the constraint
	\begin{align}\label{constraint}
	\sum_{n\geq 1}\int_{\mathbb{R}^n}\rho_n(t,y_1,...,y_n)dy_1\cdot\cdot\cdot dy_n=1.
	\end{align}
	Furthermore, for all $n\geq 2$ and $t\geq 0$ the functions
	\begin{align*}
	(y_1,...,y_n)\mapsto\rho_n(t,y_1,...,y_n)
	\end{align*}
	are symmetric in their arguments: this formalizes the indistinguishability of the different Brownian particles generated by the branching mechanism. \\
	Our first main result reads as follows.
		\begin{theorem}\label{main 1}
		The sequence $\{\rho_n\}_{n\geq 1}$ solves the system of equations
		\begin{align}\label{CDME of BBM}
		\begin{cases}
		\partial_t\rho_n(t,y_1,...,y_n)= \frac{1}{2} \sum_{k,l=1}^n\partial^2_{y_ky_l}\rho_n(t,y_1,...,y_n)+\sum_{k=1}^{n-1}\rho_k\hat{\otimes} \rho_{n-k}(t,y_1,...,y_n)\\
		\quad\quad\quad\quad\quad\quad\quad\quad\quad\mbox-\rho_n(t,y_1,...,y_n),\quad\mbox{ for } n\geq 1,t>0,y_1,...,y_n\in\mathbb{R};\\
		\rho_1(0,y_1)=\delta_{0}(y_1),\quad y_1\in\mathbb{R};\\
		\rho_n(0,y_1,...,y_n)=0,\quad y_1,...,y_n\in\mathbb{R};
		\end{cases}
		\end{align}
		here, $\delta_0$ stands for the Dirac delta function concentrated at the origin while  $\hat{\otimes}$ denotes symmetric tensor product with respect to the variables $(y_1,...,y_n)$. Moreover, we have the representation
		\begin{align}\label{c}
		\rho_n(t,y_1,...,y_n)=g_n(t,x;x-y_1,...,x-y_n),\quad t\geq 0,x,y_1,...,y_n\in\mathbb{R}
		\end{align}
		where $\{g_n\}_{n\geq 1}$ is a sequence of functions with 
		\begin{align*}
		g_n:[0,+\infty[\times\mathbb{R}\times\mathbb{R}^{n}&\to\mathbb{R}\\
		(t,x,y_1,...,y_n)&\mapsto g_n(t,x;y_1,...,y_n)
		\end{align*}
		and defined recursively as
		\begin{align*}
		\begin{cases}
		g_1(t,x;y_1):=e^{-t}\mathtt{p}_{t}(x-y_1),\quad t\geq 0,x,y_1\in\mathbb{R};\\
		g_n(t,x;y_1,...,y_n):=e^{-t}\int_0^t\int_{\mathbb{R}}\mathtt{p}_{t-s}(x-z)e^{s}\sum_{k=1}^{n-1}g_k\hat{\otimes}g_{n-k}(s,z;y_1,...,y_n)dzds,\\
		\quad\quad\quad\quad\quad\quad\quad\quad\quad\quad\quad\quad\quad n\geq 2, t\geq 0\mbox{ and } x,y_1,...,y_n\in\mathbb{R}.
		\end{cases}
		\end{align*}
		The symbol $\mathtt{p}_{t}(x-z)$ stands for the one dimensional heat kernel $(2\pi t)^{-\frac{1}{2}}e^{-\frac{(x-z)^2}{2t}}$. 
	\end{theorem}

\begin{remark}
If we view branching Brownian motion as a system of particles undergoing Brownian diffusion and splitting chemical reactions (as described above), then following \cite{delRazo},\cite{delRazo2} (see also \cite{doi1976second}) we can interpret system \eqref{CDME of BBM} within the class of chemical diffusion master equations. Equation \eqref{CDME of BBM} has been previously derived in \cite{Moyal} and further developed in \cite{Adke} but not treated as presented here. 
\end{remark}	

Now, let $\{\rho_n\}_{n\geq 1}$ be solution to \eqref{CDME of BBM} and consider for $0\leq k\leq n$ the quantity
\begin{align}\label{p}
p^{\varepsilon}_{k,n-k}(t,x):={n \choose k}\int_{B_{\varepsilon}(x)^k\times  B^c_{\varepsilon}(x)^{n-k}}\rho_n(t,y_1,...,y_n)dy_1...dy_n.
\end{align} 	
This represents the probability of having $n$ particles at time $t$ with exactly $k$ of the them in the ball centred at $x\in\mathbb{R}$ and radius $\varepsilon>0$; moreover,
writing
\begin{align}\label{X}
\mathcal{X}^{\varepsilon}(t,x):=\sum_{n\geq 1}\sum_{k=0}^nkp^{\varepsilon}_{k,n-k}(t,x)
\end{align}
we obtain the average number of particles that are located at time $t$ in the ball $B_{\varepsilon}(x)$. Lastly, the average concentration of particles at $x$ is obtained through the limit
\begin{align*}
c(t,x):=\lim_{\varepsilon\to 0}\frac{\mathcal{X}^{\varepsilon}(t,x)}{\mathtt{vol}(B_{\varepsilon}(x))}, 
\end{align*} 	
and combining \eqref{p} with \eqref{X} one finds that $\{c(t,x)\}_{t\geq 0,x\in\mathbb{R}}$, which we call \emph{average concentration field}, can be directly related to the sequence $\{\rho_n\}_{n\geq 1}$ via the formula
\begin{align}\label{ACFE from Mauricio Paper}
		c(t,x) = \sum_{n\geq 1}n\int_{\mathbb{R}^{n-1}}\rho_n(t,y_1,...,y_{n-1},x)dy_1...dy_{n-1},\quad t\geq 0,x\in\mathbb{R};
	\end{align}
(for such derivation we refer to \cite{delRazo}). We are now ready to state our second main result.

	\begin{theorem}\label{main 2}
		The average concentration field \eqref{ACFE from Mauricio Paper} associated with the system \eqref{CDME of BBM} solves the reaction-diffusion partial differential equation
		\begin{align}
		\begin{cases}\label{reaction-diffusion equation}
			\partial_tc(t,x)=\frac{1}{2}\partial_{x}^2c(t,x)+c(t,x),\quad t>0,x\in\mathbb{R};\\
			c(0,x) = \delta_{0}(x),\quad x\in\mathbb{R}.
		\end{cases}	
		\end{align}
	\end{theorem}
		
	The paper is organized as follows: Section 3 contains the proof of Theorem \ref{main 1} which is essentially based on the McKean representation for the solution to the Fisher-Kolmogorov-Petrovskii-Piskunov equation and its connection to a certain stochastic partial differential equation. Here, we also collect some notions and results from Malliavin Calculus needed for the proof of our main theorems: the implementation of these techniques represents one of the principal novelties of our contribution; then, in Section 4 we describe the proof of Theorem \ref{main 2} which is inspired by some ideas from \cite{Lanconelli_CDME}.

		\section{Proof of Theorem \ref{main 1}}
	
	It is well known \cite{McKean,BBM} that, for a continuous function $f:\mathbb{R}\to [0,1]$ and branching Brownian motion $\{X_k(t):k\leq \mathtt{n}(t)\}_{t\geq 0}$, the function
	\begin{equation}
		(t,x)\mapsto u(t,x):=\mathbb{E}\left[ \prod_{k=1}^{\mathtt{n}(t)}f(x-X_k(t)) \right]
		\label{McKean}
	\end{equation}
	is solution to the Fisher-Kolmogorov-Petrovskii-Piskunov equation \cite{fisher_paper,KPP}
	\begin{align}
		\begin{cases}\label{PDE}
			\partial_tu(t,x)=\frac{1}{2}\partial_{xx}u(t,x)+u(t,x)^2-u(t,x),& t>0,x\in\mathbb{R};\\
			u(0,x)=f(x),& x\in\mathbb{R}.
		\end{cases}
	\end{align}
	According to such representation and in view of \eqref{def rho} we can write by means of the Law of Total Probability that
	\begin{align}\label{a}
		u(t,x)=&\mathbb{E}\left[\prod_{k=1}^{\mathtt{n}(t)}f(x-X_k(t))\right]\nonumber\\
		=&\sum_{n\geq 1}\int_{\mathbb{R}^n}\mathbb{E}\left[\prod_{k=1}^{\mathtt{n}(t)}f(x-X_k(t))\Bigg|\mathtt{n}(t)=n,X(t)=y\right]\rho_n(t,y)dy\nonumber\\
		=&\sum_{n\geq 1}\int_{\mathbb{R}^n}\prod_{k=1}^{n}f(x-y_k)\rho_n(t,y)dy\nonumber\\
		=&\sum_{n\geq 1}\int_{\mathbb{R}^n}\prod_{k=1}^{n}f(y_k)\rho_n(t,x\mathtt{1}_n-y)dy\nonumber\\		
		=&\sum_{n\geq 1}\int_{\mathbb{R}^n}f^{\otimes n}(y)\rho_n(t,x\mathtt{1}_n-y)dy\nonumber\\
		=&\langle\langle \hat{U}(t,x),\mathtt{E}(f)\rangle\rangle,
	\end{align}
	where for $n\geq 1$ we used the notation $\mathtt{1}_n:=(1,1,...,1)\in\mathbb{R}^n$ and set
	\begin{align}\label{U hat}
		\hat{U}(t,x):=\sum_{n\geq 1}I_n(\rho_n(t,x\mathtt{1}_n-\cdot)),\quad t\geq 0, x\in\mathbb{R}.
	\end{align}
	In \eqref{U hat} the notation $I_n$ stands for $n$-th order multiple It\^o integral with respect to an auxiliary two sided Brownian motion, say $\{B_y\}_{y\in\mathbb{R}}$. Moreover, in \eqref{a} the brackets $\langle\langle \cdot,\cdot\rangle\rangle$ denote dual pairing between the generalized random variable $\hat{U}(t,x)$ and the \emph{stochastic exponential}  $\mathtt{E}(f)$, which is the smooth random variable defined by
	\begin{align*}
	\mathtt{E}(f):=\sum_{n\geq 0}I_n\left(\frac{f^{\otimes n}}{n!}\right)=\exp\left\{I_1(f)-\frac{1}{2}\int_{\mathbb{R}}f^2(y)dy\right\},\quad f\in C_0^{\infty}(\mathbb{R}).
	\end{align*}
 For more details on smooth and generalized random variables we refer the reader to \cite{SPDEbook} and \cite{Kuo}.\\
	We now want to relate the expression in \eqref{U hat} with the solution to a Wick-type stochastic partial differential equation investigated in \cite{uniqueness_wick}. In the sequel, the symbol $\diamond$ will denote the so-called \emph{Wick product}: given two generalized random variables $X$ and $Y$, their Wick product is defined to be the unique element $X\diamond Y$ such that
	\begin{align*}
		\langle\langle X\diamond Y,\mathtt{E}(f)\rangle\rangle=\langle\langle X,\mathtt{E}(f)\rangle\rangle\langle\langle Y,\mathtt{E}(f)\rangle\rangle,\quad\mbox{for all }f\in C_0^{\infty}(\mathbb{R}).
	\end{align*}
Alternatively, if 
\begin{align*}
	X=\sum_{n\geq 0}I_n(h_n)\quad\mbox{ and }\quad X=\sum_{n\geq 0}I_n(l_n),
\end{align*}
then 
\begin{align*}
	X\diamond Y=\sum_{n\geq 0}I_n\left(\sum_{k=0}^nh_k\hat{\otimes}l_{n-k}\right).
\end{align*}
	The transformation
	\begin{align*}
		X\mapsto\mathcal{S}(X)(f):=\langle\langle X,\mathtt{E}(f)\rangle\rangle=\sum_{n\geq 0}\int_{\mathbb{R}^n}h_n(x)f^{\otimes n}(x)dx, \quad f\in C_0^{\infty}(\mathbb{R}),
	\end{align*}
	is injective and called $\mathcal{S}$-transform of $X$.
	
	\begin{lemma}\label{lemma SPDE}
		Let $\{U(t,x)\}_{t\geq 0,x\in\mathbb{R}}$ be the unique solution to
		\begin{align}
			\begin{cases}\label{SPDE}
				\partial_tU(t,x)=\frac{1}{2}\partial^2_{x}U(t,x)+U(t,x)^{\diamond 2}-U(t,x),& t>0,x\in\mathbb{R};\\
				U(0,x)=W_x,& x\in\mathbb{R};
			\end{cases}
		\end{align}
	here, $W_x$ denotes the white noise at $x$, i.e. $W_x=I_1(\delta_0(x-\cdot))$. Then, 
		\begin{align*}
			U(t,x)=\sum_{n\geq 1}I_n(g_n(t,x;\cdot)),\quad t\geq 0,x\in\mathbb{R}
		\end{align*}
		where $\{g_n\}_{n\geq 1}$ is the sequence of functions 
		\begin{align*}
		g_n:[0,+\infty[\times\mathbb{R}\times\mathbb{R}^{n}&\to\mathbb{R}\\
		(t,x,y_1,...,y_n)&\mapsto g_n(t,x;y_1,...,y_n)
		\end{align*}
		solving the system of partial differential equations
		\begin{align}
			\begin{cases}\label{system PDE}
				\partial_tg_n(t,x;y)=\frac{1}{2}\partial^2_{x}g_n(t,x;y)+\sum_{k=1}^{n-1}g_k\hat{\otimes}g_{n-k}(t,x;y)\\
				\quad\quad\quad\quad\quad\quad\quad-g_n(t,x;y),\quad n\geq 1,t>0,x\in\mathbb{R},y\in\mathbb{R}^n;\\
				g_1(0,x;y_1)=\delta_0(x-y_1),\quad x,y_1\in\mathbb{R};\\
				g_n(0,x;y)=0,\quad n\geq 2,x\in\mathbb{R},y\in\mathbb{R}^n.
			\end{cases}
		\end{align}
	Moreover, for $f\in C_0^{\infty}(\mathbb{R})$ the function
	\begin{align*}
		(t,x)\mapsto (\mathcal{S}U(t,x))(f)=\langle\langle U(t,x),\mathtt{E}(f)\rangle\rangle
	\end{align*}
	solves equation \eqref{PDE}. 
	\end{lemma}	
	
	\begin{proof}	
		For the notion of solution to the problem \eqref{SPDE} we refer the reader to \cite{uniqueness_wick}. Projecting equation \eqref{SPDE} onto chaos spaces of different orders yields a system of partial differential equations solved by the sequence of kernels $\{g_n\}_{n\geq 1}$ of $U(t,x)$; more precisely, recalling the definition of Wick product and the fact that the initial condition in \eqref{SPDE} is made of a first order Wiener chaos one easily see that \eqref{system PDE} corresponds to the aforementioned system of partial differential equations. On the other hand, thanks to the interplay between Wick product and $S$-transform, i.e.
		\begin{align*}
			\langle\langle U(t,x)^{\diamond 2},\mathtt{E}(f)\rangle\rangle=\langle\langle U(t,x),\mathtt{E}(f)\rangle\rangle^2,\quad t\geq 0,x\in\mathbb{R},f\in C_0^{\infty}(\mathbb{R}),
		\end{align*}
	and the identity
	\begin{align*}
		\langle\langle W_x,\mathtt{E}(f)\rangle\rangle=f(x),\quad x\in\mathbb{R},f\in C_0^{\infty}(\mathbb{R}),
	\end{align*}
we see that an application of the $S$-transform to both sides of \eqref{SPDE} reduces this problem to \eqref{PDE}.
	\end{proof}	
	
	\begin{remark}
		The triangular structure of system \eqref{system PDE} allows for an explicit recursive representation of its solution; namely,
		\begin{align*}
		\begin{cases}
		g_1(t,x;y_1):=e^{-t}\mathtt{p}_{t}(x-y_1),\quad t\geq 0,x,y_1\in\mathbb{R};\\
		g_n(t,x;y_1,...,y_n):=e^{-t}\int_0^t\int_{\mathbb{R}}\mathtt{p}_{t-s}(x-z)e^{s}\sum_{k=1}^{n-1}g_k\hat{\otimes}g_{n-k}(s,z;y_1,...,y_n)dzds,\\
		\quad\quad\quad\quad\quad\quad\quad\quad\quad\quad\quad\quad\quad n\geq 2, t\geq 0\mbox{ and } x,y_1,...,y_n\in\mathbb{R}.
		\end{cases}
		\end{align*}
	\end{remark}
	
Now, employing the conclusion of Lemma \ref{lemma SPDE} in combination with \eqref{a} we can affirm that
	\begin{align*}
		\langle\langle \hat{U}(t,x),\mathtt{E}(f)\rangle\rangle=\langle\langle U(t,x),\mathtt{E}(f)\rangle\rangle\quad\mbox{for all $t\geq 0$, $x\in\mathbb{R}$ and $f\in C_0^{\infty}(\mathbb{R})$};
	\end{align*}
	from the invertibility of the $S$-transform and uniqueness of chaos expansion's kernels we then deduce
	\begin{align*}
		\rho_n(t,x\mathtt{1}_n-y)=g_n(t,x;y),\quad\mbox{ for all }n\geq 1, t\geq 0, x\in\mathbb{R},y\in\mathbb{R}^n
	\end{align*}
	and thus
	\begin{align*}
		\rho_n(t,y_1,...,y_n)=&\rho_n(t,x-(x-y_1),...,x-(x-y_n))\\
		=&g_n(t,x;x-y_1,...,x-y_n),\quad n\geq 1,t\geq 0,x,y_1,...,y_n\in\mathbb{R},
	\end{align*}	
	which corresponds to \eqref{c}. Lastly, recalling that the sequence $\{g_n\}_{n\geq 1}$ solves the system of equations \eqref{system PDE} we have
	\begin{align*}
		\partial_t\rho_n(t,x\mathtt{1}_n-y)=&\partial_tg_n(t,x;y)\\
		=&\frac{1}{2}\partial^2_{x}g_n(t,x;y)+\sum_{k=1}^{n-1}g_k\hat{\otimes}g_{n-k}(t,x;y)-g_n(t,x;y)\\
		=&\frac{1}{2}\partial^2_{x}(\rho_n(t,x\mathtt{1}_n-y))+\sum_{k=1}^{n-1}\rho_k\hat{\otimes}\rho_{n-k}(t,x\mathtt{1}_n-y)-\rho_n(t,x\mathtt{1}_n-y)\\
		=&\frac{1}{2}\sum_{j,k=1}^n\partial^2_{y_jy_k}\rho_n(t,x\mathtt{1}_n-y)+\sum_{k=1}^{n-1}\rho_k\hat{\otimes}\rho_{n-k}(t,x\mathtt{1}_n-y)-\rho_n(t,x\mathtt{1}_n-y);
	\end{align*}
if we now replace $x\mathtt{1}_n-y$ with $y$ we obtain the desired system \eqref{CDME of BBM}, completing the proof of Theorem \ref{main 1}.

	\begin{remark}	
	One can see the interplay between the different Brownian particles generated by the branching mechanism in the diffusion term
	\begin{align*} 
	\frac{1}{2} \sum_{k,l=1}^n\partial^2_{y_ky_l}\rho_n(t,y_1,...,y_n).
	\end{align*}
	This is due to the fact that the newly created particles do not emerge at an independent random location but enter to the system at the same exact location of the parent particle. Such feature is in contrast with the models considered in \cite{delRazo} and \cite{delRazo2} where diffusion is described by a simple Laplace operator.
\end{remark}	

\section{Proof of Theorem \ref{main 2}}
	
	Inspired by the ideas presented in \cite{Lanconelli_CDME} and further developed in  \cite{bdCDME} and \cite{mutual_annihilation_paper} we now rewrite the system of equations \eqref{CDME of BBM} as a single abstract equation containing Malliavin Calculus operators. For details on the subject the reader is referred to the books \cite{Hubook}, \cite{Janson} and \cite{Nualart}; here, we recall few basic definitions and properties  needed for proving our result. In the sequel, we denote by $\{D_x\Phi\}_{x\in\mathbb{R}}$ the \emph{Malliavin derivative} of $\Phi=\sum_{n\geq 0}I_n(h_n)$ defined as
	\begin{align*}
		D_x\Phi:=\sum_{n\geq 1}nI_{n-1}(h_{n}(\cdot,x)), \quad x\in \mathbb{R};
	\end{align*}
	notice that
	\begin{align*}
		D_x\mathtt{E}(f)=f(x)\mathtt{E}(f),\quad\mbox{ for all }f\in C_0^{\infty}(\mathbb{R}).
	\end{align*} 
	Moreover, for a possibly unbounded $A:L^2(\mathbb{R})\to L^2(\mathbb{R})$ we define its \emph{differential second quantization operator} as
	\begin{align*}
		d\Gamma(A)\Phi:=\sum_{n\geq 1}I_n\left(\sum_{i=1}^nA_ih_n\right),
	\end{align*}
	where $A_i$ stands for the operator $A$ acting on the $i$-th variable of $h_n$. 
	The following useful identities hold true:
	\begin{align*}
		\begin{split}
			&\langle\langle d\Gamma(A)\Phi,1\rangle\rangle=0;\quad\langle\langle d\Gamma(A)\Phi,\Psi\rangle\rangle=\langle\langle\Phi, d\Gamma(A^{\star})\Psi\rangle\rangle;\\
			&d\Gamma(A)\mathtt{E}(f)=\mathtt{E}(f)\diamond I_1(Af);\quad D_x(\Phi\diamond \Psi)=D_x\Phi\diamond \Psi+\Phi\diamond D_x\Psi.
		\end{split}
	\end{align*}

	\begin{lemma}
		Let $\{\rho_n\}_{n\geq 1}$ be solution to \eqref{CDME of BBM}. Then, the generalized stochastic process
		\begin{align}\label{Phi def}
		\Phi(t):=\sum_{n\geq 0}I_n(\rho_n(t,\cdot)),\quad t\geq 0,
	    \end{align}
		solves the abstract equation 
		\begin{align}
			\begin{cases}\label{Phi equation}
			\partial_t\Phi(t) = \frac{1}{2}d\Gamma(\partial)^2\Phi(t)+\Phi(t)^{\diamond2}-\Phi(t),\quad t>0;\\
			\Phi(0)=W_0.
			\end{cases}
		\end{align}
	\end{lemma}
	
	\begin{proof}
	We have
	\begin{align*}
		\partial_t\Phi(t)=&\partial_t\sum_{n\geq 0}I_n(\rho_n(t,\cdot))\\
		=&\sum_{n\geq 0}I_n(\partial_t\rho_n(t,\cdot))\\
		=&\sum_{n\geq 0}I_n\left(\frac{1}{2} \sum_{k,l=1}^n\partial^2_{y_ky_l}\rho_n(t,\cdot)\right)+\sum_{n\geq 0}I_n\left(\sum_{k=1}^{n-1}\rho_k\hat{\otimes} \rho_{n-k}(t,\cdot)\right)\\
		&-\sum_{n\geq 0}I_n(\rho_n(t,\cdot))\\
		=&\frac{1}{2}\sum_{n\geq 0}d\Gamma(\partial)^2I_n\left(\rho_n(t,\cdot)\right)+\sum_{n\geq 0}I_n(\rho_n(t,\cdot))\diamond\sum_{n\geq 0}I_n(\rho_n(t,\cdot))\\
		&-\sum_{n\geq 0}I_n(\rho_n(t,\cdot))\\
		=&\frac{1}{2}d\Gamma(\partial)^2\Phi(t)+\Phi(t)^{\diamond2}-\Phi(t).
	\end{align*}
	Here, we utilized \eqref{CDME of BBM} and the definitions of differential second quantization operator and Wick product. Moreover,
	\begin{align*}
		\Phi(0)=\sum_{n\geq 0}I_n(\rho_n(0,\cdot))=I_1(\delta_0)=W_0.
	\end{align*}
	\end{proof}

\begin{remark}\label{w}
	Notice that using the notation introduced in the previous lemma condition \eqref{constraint} reads 
	\begin{align*}
	\langle\langle \Phi(t),\mathtt{E}(1)\rangle\rangle=1,\quad\mbox{ for all }t\geq 0.
	\end{align*}
\end{remark}
	
	The usefulness of switching from system \eqref{CDME of BBM} to equation \eqref{Phi equation} becomes evident in the following result where a compact representation for the average concentration field is proposed.
	
	\begin{lemma}
		The average concentration field \eqref{ACFE from Mauricio Paper} enjoys the representation
	\begin{align}\label{dd}
	c(t,x)=\langle\langle D_x\Phi(t),\mathtt{E}(1)\rangle\rangle,\quad t\geq 0, x\in\mathbb{R}.
	\end{align}
	\end{lemma}

\begin{proof}
	From \eqref{Phi def} we can write
	\begin{align*}
		D_x\Phi(t)=D_x\sum_{n\geq 0}I_n(\rho_n(t,\cdot))=\sum_{n\geq 1}nI_{n-1}(\rho_n(t,\cdot,x));
	\end{align*}
	therefore,
	\begin{align*}
		\langle\langle D_x\Phi(t),\mathtt{E}(1)\rangle\rangle=\sum_{n\geq 1}n\int_{\mathbb{R}^{n-1}}\rho_n(t,y_1,...,y_{n-1},x)dy_1...dy_{n-1},
	\end{align*}
and since the last member above agrees with \eqref{ACFE from Mauricio Paper} the validity of formula \eqref{dd} is proven.
\end{proof} 
		
We are now ready to prove Theorem \ref{main 2}. By means of the two previous lemmas we get
\begin{align}\label{z}
	\partial_tc(t,x)=&\partial_t\langle\langle D_x\Phi(t),\mathtt{E}(1)\rangle\rangle\nonumber\\
	=&\langle\langle \partial_tD_x\Phi(t),\mathtt{E}(1)\rangle\rangle\nonumber\\
	=&\langle\langle D_x\partial_t\Phi(t),\mathtt{E}(1)\rangle\rangle\nonumber\\
	=&\langle\langle D_x\left(\frac{1}{2}d\Gamma(\partial)^2\Phi(t)+\Phi(t)^{\diamond2}-\Phi(t)\right),\mathtt{E}(1)\rangle\rangle\nonumber\\
	=&\frac{1}{2}\langle\langle D_xd\Gamma(\partial)^2\Phi(t),\mathtt{E}(1)\rangle\rangle+\langle\langle D_x\Phi(t)^{\diamond2},\mathtt{E}(1)\rangle\rangle-\langle\langle D_x\Phi(t),\mathtt{E}(1)\rangle\rangle.
\end{align}		
Moreover,
\begin{itemize}		
	\item the following commutation relation holds:
	\begin{align}\label{comm dGamma}	D_xd\Gamma(\partial)=\left(d\Gamma(\partial)+\frac{d}{dx}\right)D_x.	
	\end{align}
	In fact, checking its validity on stochastic exponentials we get 	 
	\begin{align*}
	d\Gamma(\partial)D_x\mathtt{E}(h)=d\Gamma(\partial)\mathtt{E}(h) h(x)= h(x)\mathtt{E}(h)\diamond I_1(h')
	\end{align*}
	and 
	\begin{align*}
	D_xd\Gamma(\partial)\mathtt{E}(h)=D_x \mathtt{E}(h)\diamond I_1(h')= \mathtt{E}(h)h(x)\diamond I_1(h')+\mathtt{E}(h)h';
	\end{align*}
	thus,
	\begin{align*}
	D_xd\Gamma(\partial)\mathtt{E}(h)&=d\Gamma(\partial)D_x\mathtt{E}(h)+\mathtt{E}(h)h'\\
	&= d\Gamma(\partial)D_x\mathtt{E}(h)+  \frac{d}{dx}D_x\mathtt{E}(h);
	\end{align*}
\item iterating \eqref{comm dGamma} twice one gets
\begin{align}\label{comm double}
	D_xd\Gamma(\partial)^2=\left(d\Gamma(\partial)^2+2d\Gamma(\partial)\frac{d}{dx}+\frac{d^2}{dx^2}\right)D_x;
\end{align}
\item for all $\Phi=\sum_{n\geq 0}I_n(h_n)$ we have
\begin{align}\label{vacuum}
\langle\langle d\Gamma(\partial)\Phi,\mathtt{E}(1)\rangle\rangle=0.
\end{align}
In fact, using the properties of differential second quantization operators we can write
\begin{align*}
\langle\langle d\Gamma(\partial)\Phi,\mathtt{E}(1)\rangle\rangle=\langle\langle \Phi,d\Gamma(-\partial)\mathtt{E}(1)\rangle\rangle=\langle\langle \Phi,\mathtt{E}(1)\diamond I_1(0)\rangle\rangle=0.
\end{align*}
\end{itemize}
By virtue of \eqref{comm double} and \eqref{vacuum} the first term in \eqref{z} can be simplified to
\begin{align*}
	\frac{1}{2}\langle\langle D_xd\Gamma(\partial)^2\Phi(t),\mathtt{E}(1)\rangle\rangle&=\frac{1}{2}\langle\langle \frac{d^2}{dx^2}D_x\Phi(t),\mathtt{E}(1)\rangle\rangle\\
	&=\frac{1}{2}\frac{d^2}{dx^2}\langle\langle D_x\Phi(t),\mathtt{E}(1)\rangle\rangle\\
	&=\frac{1}{2}\frac{d^2}{dx^2}c(t,x).
\end{align*}
The second term in \eqref{z}, exploiting the chain rule for the Malliavin derivative and Wick product, becomes
\begin{align*}
	\langle\langle D_x\Phi(t)^{\diamond2},\mathtt{E}(1)\rangle\rangle&=2\langle\langle \Phi(t)\diamond D_x\Phi(t),\mathtt{E}(1)\rangle\rangle\\
	&=2\langle\langle \Phi(t),\mathtt{E}(1)\rangle\rangle\langle\langle D_x\Phi(t),\mathtt{E}(1)\rangle\rangle\\
	&=2\langle\langle D_x\Phi(t),\mathtt{E}(1)\rangle\rangle\\
	&=2c(t,x),
\end{align*} 
where we also utilized the observation in Remark \ref{w}. Combining those identities in \eqref{z} we conclude that
\begin{align*}
	\partial_tc(t,x)=\frac{1}{2}\partial_x^2c(t,x)+c(t,x).
\end{align*}
Moreover,
\begin{align*}
c(0,x)=\langle\langle D_x\Phi(0),\mathtt{E}(1)\rangle\rangle=\langle\langle D_xW_0,\mathtt{E}(1)\rangle\rangle=\delta_{0}(x).
\end{align*}
The proof of Theorem \ref{main 2} is now complete.	

\bibliographystyle{plain}
\bibliography{bibliography}

\begin{thebibliography}{10}

\bibitem{Adke}
S.~R. Adke.
\newblock The generalized birth and death process and {G}aussian diffusion.
\newblock {\em J. Math. Anal. Appl.}, 9:336--340, 1964.

\bibitem{Moyal}
S.~R. Adke and J.~E. Moyal.
\newblock A birth, death, and diffusion process.
\newblock {\em J. Math. Anal. Appl.}, 7:209--224, 1963.

\bibitem{BBM}
Anton Bovier.
\newblock {\em From spin glasses to branching Brownian motion - and back?},
  volume 2144.
\newblock 08 2015.

\bibitem{delRazo}
Mauricio~J. del Razo, Daniela Fr\"{o}mberg, Arthur~V. Straube, Christof
  Sch\"{u}tte, Felix H\"{o}fling, and Stefanie Winkelmann.
\newblock A probabilistic framework for particle-based reaction-diffusion
  dynamics using classical {F}ock space representations.
\newblock {\em Lett. Math. Phys.}, 112(3):Paper No. 49, 59, 2022.

\bibitem{delRazo2}
Mauricio~J del Razo, Stefanie Winkelmann, Rupert Klein, and Felix H{\"o}fling.
\newblock Chemical diffusion master equation: Formulations of
  reaction--diffusion processes on the molecular level.
\newblock {\em J. Math. Phys.}, 64(1):013304, 2023.

\bibitem{doi1976second}
Masao Doi.
\newblock Second quantization representation for classical many-particle
  system.
\newblock {\em Journal of Physics A: Mathematical and General}, 9(9):1465,
  1976.

\bibitem{fisher_paper}
R.~A. Fisher.
\newblock The wave of advance of advantageous genes.
\newblock {\em Annals of Eugenics}, 7(4):355--369, 1937.

\bibitem{SPDEbook}
Helge Holden, Bernt \O~ksendal, Jan Ub\o~e, and Tusheng Zhang.
\newblock {\em Stochastic partial differential equations}.
\newblock Probability and its Applications. Birkh\"{a}user Boston, Inc.,
  Boston, MA, 1996.
\newblock A modeling, white noise functional approach.

\bibitem{Hubook}
Yaozhong Hu.
\newblock {\em Analysis on {G}aussian spaces}.
\newblock World Scientific Publishing Co. Pte. Ltd., Hackensack, NJ, 2017.

\bibitem{Janson}
Svante Janson.
\newblock {\em Gaussian {H}ilbert spaces}, volume 129 of {\em Cambridge Tracts
  in Mathematics}.
\newblock Cambridge University Press, Cambridge, 1997.

\bibitem{KPP}
A.~Kolmogorov, I~Petrovsky, and N.~Piscounov.
\newblock Etude de l’ équation de la diffusion avec croissance de la
  quantité de matière et son application à un problème biologique.
\newblock {\em Bull. Math.}, 1:1--25, 1937.

\bibitem{Kuo}
Hui-Hsiung Kuo.
\newblock {\em White noise distribution theory}.
\newblock Probability and Stochastics Series. CRC Press, Boca Raton, FL, 1996.

\bibitem{Lanconelli_CDME}
Alberto Lanconelli.
\newblock Using malliavin calculus to solve a chemical diffusion master
  equation.
\newblock {\em Journal of Mathematical Analysis and Applications},
  526(2):127352, 2023.

\bibitem{mutual_annihilation_paper}
Alberto Lanconelli and Berk~Tan Perçin.
\newblock Probabilistic derivation and analysis of the chemical diffusion
  master equation with mutual annihilation.
\newblock {\em arXiv2306.05139}, 2023.

\bibitem{bdCDME}
Alberto Lanconelli, Berk~Tan Perçin, and Mauricio~J. del Razo.
\newblock Solution formula for the general birth-death chemical diffusion
  master equation.
\newblock {\em arXiv2302.10700}, 2023.

\bibitem{uniqueness_wick}
Tijana Levajković, Stevan Pilipović, Dora Seleši, and Milica Žigić.
\newblock {Stochastic evolution equations with Wick-polynomial nonlinearities}.
\newblock {\em Electronic Journal of Probability}, 23:1 -- 25, 2018.

\bibitem{McKean}
H.~P. McKean.
\newblock Application of {B}rownian motion to the equation of
  {K}olmogorov-{P}etrovskii-{P}iskunov.
\newblock {\em Comm. Pure Appl. Math.}, 28(3):323--331, 1975.

\bibitem{Nualart}
David Nualart.
\newblock {\em The {M}alliavin calculus and related topics}.
\newblock Probability and its Applications (New York). Springer-Verlag, Berlin,
  second edition, 2006.

\end{thebibliography}
	
\end{document}